\documentclass[11pt]{amsart}
\usepackage{amssymb, amsmath, amsthm}
\usepackage[latin1]{inputenc}
\usepackage{graphicx}
\usepackage[final]{hyperref}
\usepackage[a4paper, centering]{geometry}
\newtheorem{theorem}{Theorem}
\newtheorem{proposition}[theorem]{Proposition}
\newtheorem{lemma}[theorem]{Lemma}

\theoremstyle{definition}
\newtheorem{definition}[theorem]{Definition}
\theoremstyle{remark}
\newtheorem{remark}[theorem]{Remark}
\newtheorem{example}[theorem]{Example}
\def\R{\mathbb{R}}

\def\N{\mathbb{N}}
\def\haus{\mathcal{H}^{n-1}}
\def\pscal#1#2{\left\langle#1,\,#2\right\rangle}

\def\ff{f}
\def\ll{\lambda}

\DeclareMathOperator{\dive}{div}
\def\W{\mathcal W}

\begin{document}

\title[A new symmetry criterion]%
{A new symmetry criterion\\ based on the distance function \\ and applications to PDE's}%
\author[G.~Crasta, I.~Fragal\`a]{Graziano Crasta,  Ilaria Fragal\`a}
\address[Graziano Crasta]{Dipartimento di Matematica ``G.\ Castelnuovo'', Univ.\ di Roma I\\
P.le A.\ Moro 2 -- 00185 Roma (Italy)}

\email{crasta@mat.uniroma1.it}

\address[Ilaria Fragal\`a]{
Dipartimento di Matematica, Politecnico\\
Piazza Leonardo da Vinci, 32 --20133 Milano (Italy)
}

\email{ilaria.fragala@polimi.it}



\date{\today}

\begin{abstract} We prove that, if $\Omega\subset \R ^n$ is an open bounded starshaped domain of class $C^2$, the constancy over $\partial \Omega$
of the function
$$\varphi(y) = \int_0^{\ll(y)} \prod_{j=1}^{n-1}[1-t \kappa_j(y)]\, dt$$
implies that $\Omega$ is a ball. Here $k _j(y)$ and $\lambda(y)$ denote respectively the principal curvatures and the cut value of a boundary point $y \in \partial \Omega$.
We apply this geometric result to different symmetry questions for PDE's: an overdetermined system of Monge-Kantorovich type equations
(which can be viewed as the limit as $p \to + \infty$ of Serrin's symmetry problem for the $p$-Laplacian), and equations in divergence form whose solutions depend only on the distance from the boundary in some subset of their domain.
\end{abstract}

\maketitle

\section{Introduction}
Characterizing special classes of hypersurfaces in a metric space, in particular spheres, in terms of some
properties of their principal curvatures,  is a classical and challenging problem in Differential Geometry.
A fundamental result by Alexandrov states that a bounded smooth domain in the Euclidean space is a ball provided the mean curvature
of its boundary is constant \cite{Al}.  For further characterizations of spheres involving the symmetric functions of the principal curvatures, see {\it e.g.}
\cite{Ko, MoRo}  and the references therein.

Alexandrov's result has many powerful applications in Analysis, especially in the fields of PDE's and shape optimization; in fact,
it allows to obtain for instance symmetry in overdetermined boundary value problems (as in the seminal paper \cite{S} and the subsequent literature),
of in extremum problems for variational functionals under geometric constraints (see the monograph \cite{He}).

Often, symmetry questions arise in problems in which a crucial role is played by the distance function from the boundary of an open bounded  domain $\Omega\subset \R ^n$ , $d_\Omega (x) := {\rm dist} (x, \partial \Omega)$. This happens for instance when studying
PDE's related with mass transportation theory (see \cite{BoBu, CaCa, CCCG}), or
minimization problems in the class of so-called {\it web functions}, namely functions which only depend on $d_\Omega$ (see \cite{C, CFG1, CFG2, CFG3, CG1, CG2, G}).
Symmetry questions in these frameworks, which will be described more precisely below, pushed us to set up a new roundedness criterion, which brings into play
the distance function in a more intrinsic way than merely through the boundary curvatures.  More precisely, it involves the following function $\varphi$ associated with a bounded smooth domain $\Omega$ of $\R^n$:
\begin{equation}\label{f:phi}
\varphi(y) := \int_0^{\ll(y)} \prod_{j=1}^{n-1}[1-t \kappa_j(y)]\, dt\,,
\qquad y\in\partial\Omega\,.
\end{equation}
Here $\kappa _j(\cdot) $ denote the principal curvatures of $\partial \Omega$, whereas $\lambda (y)$ is the {\it cut value} of $y$: letting $\nu(y)$ denote the unit outer normal to $\partial \Omega$ at $y$, and  $\pi (x)$ be the point of $\partial \Omega$ such that $|x - \pi (x)| = {\rm dist} (x, \partial \Omega)$ (which is uniquely determined for ${\mathcal L} ^n$-a.e.\ $x \in \Omega$), the function $\lambda ( \cdot)$ is  defined on $\partial \Omega$ by
\begin{equation}\label{cut}
\lambda (y) := \sup \big \{ t \geq 0 \ :\ \pi (y - t \nu (y) ) = y \big \}\ , \qquad  y \in \partial \Omega\ .
\end{equation}
Intuitively, by following the inner normal to $\partial \Omega$ at $y$, one can continue without crossing another straight line normal to $\partial \Omega$, exactly until one arrives at a distance $\lambda (y)$ from the boundary. Thus, $\Omega$ is filled up by the line segments $\{ y - t \nu (y) \, :\, y \in \partial \Omega\, ,\,  t \in (0, \lambda (y))\}$.

Let us mention that the functions $\varphi$ and $\lambda$ already appeared in the literature in different contexts: concerning the function $\varphi$, it is a crucial tool in the proof of the isoperimetric inequality \textsl{\`a la Gromov} (see {\it e.g.} \cite[\S 1.6.8]{Be}), and appears also in mathematical models for granular materials \cite{CCCG, CCS,CM}; concerning the function $\lambda$, its regularity has been studied in \cite{CCG,IT,LiNi}, while some of its applications to variational problems can be found in \cite{Ce1,CFG1,CFG2,CFG3}.

Our new symmetry criterion is formulated  in terms of
the function $\varphi$ and of the mean curvature $H$ of $\partial \Omega$,
\begin{equation}\label{mc}
H(y) := \frac{\kappa_1(y)+\cdots\kappa_{n-1}(y)}{n-1}\,,
\qquad y\in\partial\Omega\,,
\end{equation}
and reads  as follows. By saying that $\Omega$ is {\it starshaped}, we mean that
\begin{equation}\label{star}
\langle y, \nu (y)  \rangle >0 \qquad \forall \, y \in \partial \Omega \ .
\end{equation}

\begin{theorem}\label{t:geom2}
Let $\Omega\subset\R^n$ be a bounded connected open set of class $C^2$,
starshaped with respect to the origin.
Assume that
there exists a point $y_0\in\partial\Omega$
such that
\begin{equation}\label{f:hypoa}
H(y_0) = \max_{y\in\partial\Omega} H(y)\,,\qquad
\varphi(y_0) \geq \frac{|\Omega|}{|\partial\Omega|}
\,.
\end{equation}
Then $\Omega$ is a ball. In particular, if $\varphi$ is constant on $\partial \Omega$, assumption \eqref{f:hypoa} is satisfied, and hence $\Omega$ is a ball.
\end{theorem}

The proof of Theorem \ref{t:geom2} is given in Section \ref{proof} below. It is obtained by showing that the existence of a point $y _0\in \partial \Omega$ fitting the two conditions in (\ref{f:hypoa}) ensures the constancy of the mean curvature. To that aim, we exploit as crucial tools the arithmetic-geometric inequality  and the Minkowski integral formula for the mean curvature (recalled in  Section \ref{secnot} below).
In particular, in order to apply successfully such formula, we need the starshapedness condition (\ref{star}):
we believe it may be unnecessary for the validity of the result, but at present we are unable to circumvent it.

Concerning the $C^2$-regularity assumption,  in dimension $n =2$ we are able to relax it, by allowing domains whose boundary is 
piecewise $C^2$ and may contain ``convex'' corners 
(see Section \ref{proof} for more details).

\medskip
Let us now turn attention to describe two different  applications of Theorem \ref{t:geom2} to symmetry questions for PDE's.

\medskip
The first application concerns an overdetermined system of PDE's of Monge-Kantorovich type, which can be viewed a limit version of Serrin's symmetry result for the $p$-Laplacian operator, as $p$ tends to $+ \infty$.
To be more precise, let us recall Serrin's result: existence of a solution to the overdetermined boundary value problem
\begin{equation}\label{deqse} \left\{\begin{array}{rll}
- \Delta u =1\qquad&\hbox{ in }\Omega\, ,\\
u=0\qquad&\hbox{ on }\partial\Omega\, ,\\
|\nabla u| =c\qquad&\hbox{ on }\partial\Omega\, ,
\end{array}\right.
\end{equation}
where $c$ is a positive constant, implies that $\Omega$ is a ball.
Later on, the same symmetry statement has been generalized to the case when the Laplacian in (\ref{deqse}) is replaced by more general operators; since
the literature on this topic is very broad, we limit ourselves to quote the paper \cite{FGK}, where the interested reader can also find many related references. In particular, Serrin's result extends to the system
\begin{equation}\label{deqp} \left\{\begin{array}{rll}
- \Delta_p u =1\qquad&\hbox{ in }\Omega\,,\\
u=0\qquad&\hbox{ on }\partial\Omega\,,\\
|\nabla u| =c\qquad&\hbox{ on }\partial\Omega\,,
\end{array}\right.
\end{equation}
where $\Delta _p$ denotes the
$p$-Laplacian operator (see \cite{[BH],[DP],[GL]}).

It is then natural to ask whether the same result continues to be true when considering, in some  sense,  the limit as $p \to + \infty$.
Let us mention that, in this spirit, overdetermined boundary value problems for the $\infty$-Laplacian have been recently studied in \cite{BuKa}.
Here we look rather at the system of PDE's of Monge-Kantorovich type which arises as the limiting problem of (\ref{deqp}) when  $p \to + \infty$.
Actually, let $u _p$ be the unique solution to the Dirichlet boundary value problem given by the first two equations in (\ref{deqp}).
As explained more in detail in Section  \ref{secsabbia}, passing to the limit as $p \to + \infty$ in such Dirichlet problem, leads to consider the following system in the unknowns $u \in {\rm Lip } (\Omega) $ and $v \in L^1(\Omega; \R^+)$:
\begin{equation}\label{f:MK}
\begin{cases}
-\dive (v \nabla u) = 1 & \text{in}\ \Omega\,,\\
|\nabla u| \leq 1 & \text{a.e.\ in}\ \Omega\,, \\
u=0  &  \text{on}\ \partial\Omega\,,\\
(1-|\nabla u|) v = 0 & \text{a.e.\ in}\ \Omega\,. \\
\end{cases}
\end{equation}
Heuristically, the variables $u$ and $v$ represent respectively the limit of $u _p$ and $|\nabla u _p| ^ p$ as $p \to + \infty$
({\it cf.}\ Section \ref{secsabbia}).
System (\ref{f:MK}) always admits solutions;
moreover, if $(u,v)$ is a solution to (\ref{f:MK}),
then $v\in C(\overline{\Omega})$ {see \cite{CaCa,CCCG}).
Therefore, the following question makes sense:
does symmetry holds for (\ref{f:MK}) if, as a counterpart to the last equation in (\ref{deqp}), we ask that
$v$ is constant on $\partial\Omega$?
In other words:
\begin{equation}\label{Q2}
\textit{If   \eqref{f:MK} admits a solution $(u, v)$ with $v$ constant on $\partial\Omega$, is $\Omega$ a ball?}
\end{equation}

In Section \ref{secsabbia}, we show that the answer to question (\ref{Q2}) is affirmative, as a consequence of Theorem \ref{t:geom2}.
We provide two significant physical interpretations of this symmetry result in different frameworks:
a shape optimization problem for heat conductors considered in \cite{BoBu}, and a two-layers model in granular matter theory studied in \cite{CCCG}.

\medskip
The second application concerns ``partially web solutions'' to equations
in divergence form.
Let $\Omega$ be as above, and let $\omega$ be a bounded connected open subset of $\Omega$;
we define the space of web functions on $\Omega$ and the space of their restrictions to $\omega$  respectively as
$$\begin{array}{ll}
& \W (\Omega) := \Big \{ u \in W ^ {1, 1} (\Omega) \ : \ u(x) \hbox{ depends only from } d_\Omega (x) \Big \}\,,
\\ \noalign{\medskip}
& \W (\Omega; \omega) := \Big \{ u \in W ^ {1, 1} (\Omega) \ : \   u (x) = \tilde u(x)  \,  \hbox{ a.e.\ on $\omega$, for some } \tilde u \in \W (\Omega)  \Big \}\ .
\end{array}
$$
Given an equation in divergence form on $\Omega$, with a constant source term,
\begin{equation}\label{f:elleq}
- {\rm div}\big ( A(|\nabla u|) \nabla u\big ) = 1 \qquad \mbox{ in } \Omega\ ,
\end{equation}
where $A\in C([0,+\infty))$,
we say that $u$ is a solution if $A(|\nabla u|)\, |\nabla u| \in L^1(\Omega)$ and
$$
\int_{\Omega} A(|\nabla u|) \pscal{\nabla u}{\nabla\psi} \, dx
= \int_{\Omega} \psi\, dx
\qquad
\forall \psi\in C^{\infty}_0(\Omega)\,.
$$
We then ask the following question:
\begin{equation}\label{Q1}
\textit{If \eqref{f:elleq} admits a solution $u$ belonging to $\W (\Omega; \omega)$, is $\Omega$ a ball?}
\end{equation}
In order to provide some conditions on $\omega$ which are sufficient for a positive answer to question (\ref{Q1}), we shall restrict attention to the case when $\omega$ is of the form
$$\Omega _\Gamma:= \Big \{ y - t \nu (y) \ :\ y \in \Gamma \, ,\ t \in (0, \lambda (y)) \Big \}\ ,$$
for some relatively open connected set  $\Gamma \subseteq \partial \Omega$.

Let us remark that, if $u$ is a solution to $(\ref{f:elleq})$ belonging to $\W (\Omega; \Omega_\Gamma)$, since the restriction of $u$ to $\Omega _\Gamma$ can be written as $h (d _\Omega)$ for some function $h$, it holds
$$u = h (0) \qquad \hbox{ and } \qquad |\nabla u| = |h' (0)| \qquad \hbox{ on } \Gamma\ .$$
Hence, up to an additive constant, a solution $u$ to $(\ref{f:elleq})$ lying in $\W (\Omega; \Omega_\Gamma)$ satisfies the following system
(which is clearly weaker than the requirement $u \in \W (\Omega; \Omega_\Gamma)$):
\begin{equation}\label{deqbis} \left\{\begin{array}{rll}
-{\rm div}(A(|\nabla u|)\nabla u)=1\qquad&\hbox{ in }\Omega\,,\\
u=0\qquad&\hbox{ on }\Gamma\,, \\
|\nabla u| =c\qquad&\hbox{ on }\Gamma\, .
\end{array}\right.
\end{equation}
In particular,  if $\Gamma \equiv \partial \Omega$, the answer to question (\ref{Q1}) is affirmative, as soon as the operator $A$ is such that Serrin's symmetry result is valid for the elliptic equation (\ref{f:elleq});  in this direction, let us  mention that symmetry for minimizers to variational functionals in $\W(\Omega)$ was studied in \cite{C}.

Here we are rather interested in the case when $\Gamma$ is strictly contained into $\partial \Omega$. We point out that
system (\ref{deqbis}) is quite close to the the so-called partially overdetermined boundary value problems studied for the Laplace operator in \cite{FG}. However, in such partially overdetermined problems, one among the Dirichlet and the Neumann condition was required to hold on the whole of $\partial \Omega$ (see also \cite{FV, FGLP}), which is not necessarily the case when
$u \in \W (\Omega; \Omega_\Gamma)$.

In
Section \ref{secparweb}, using Theorem \ref{t:geom2}, we obtain sufficient conditions on $\Gamma$ for a positive answer to question (\ref{Q1}) in  dimension $n=2$. We also compare more in detail
our results and those obtained for partially overdetermined boundary value problems in \cite{FG}.

\medskip
The paper is organized as follows.
After providing some preliminaries in Section \ref{secnot}, in  Section \ref{proof} we give the proof of Theorem \ref{t:geom2}, and we generalize it to piecewise $C^2$ domains in the two-dimensional case.
Sections \ref{secsabbia} and \ref{secparweb} are  devoted to the applications of the geometric result to PDE's,
 respectively to Monge-Kantorovich type equations, and to
equations in divergence form with partially web solutions.

\bigskip
{\bf Acknowledgments.}
The authors wish to acknowledge the helpful comments
of Giuseppe Buttazzo and Filippo Gazzola
during the preparation of the manuscript.


\section{Notation and Preliminaries}\label{secnot}

Let us fix some notation and geometric background.

The standard scalar product of two vectors $x,y\in\R^n$
is denoted by $\pscal{x}{y}$, and $|x|$ denotes the
Euclidean norm of $x\in\R^n$.
As is customary, $B_r(x_0)$ and $\overline{B}_r(x_0)$
are respectively the open and the closed ball
centered at $x_0$ and with radius $r>0$.

If $\Omega$ is an open subset of $\R^n$, we shall denote
by $|\Omega|$ and $|\partial\Omega|$ respectively the $n$-dimensional Lebesgue measure of $\Omega$
and the $(n-1)$-dimensional Hausdorff measure of its boundary.

A bounded open set $\Omega\subset\R^n$
(or, equivalently, its closure
$\overline{\Omega}$ or its boundary $\partial \Omega$)
is of class $C^k$, $k\in\N$,
if for every point $x_0\in\partial \Omega$
there exists a ball $B=B_r(x_0)$ and a one-to-one
mapping $\psi\colon B\to D$ such that
$\psi\in C^k(B)$, $\psi^{-1}\in C^k(D)$,
$\psi(B\cap \Omega)\subseteq\{x\in\R^n;\ x_n > 0\}$,
$\psi(B\cap\partial \Omega)\subseteq\{x\in\R^n;\ x_n = 0\}$.

Given a nonempty open set $\Omega\subset\R^n$, we denote by
$d_{\Omega}\colon\overline{\Omega}\to\R$ the distance function from the boundary of $\Omega$,
defined by
\[
d_{\Omega}(x) := \min_{y\in\partial\Omega} |x-y|,\quad x\in\overline{\Omega}.
\]
It is well-known that $d_{\Omega}$ is a $1$-Lipschitz function; by Rademacher theorem,
its singular set $\Sigma$ ({\it i.e.}, the set of points where $d_{\Omega}$ is not differentiable)
has zero Lebesgue measure
(more precisely, it is a $\mathcal{H}^{n-1}$-rectifiable set,
see \cite[Prop.~4.1.3]{CaSi}).
The closure of $\Sigma$ is called the \textsl{cut locus} of $\Omega$, and in general
it may have a non-vanishing Lebesgue measure.
Nevertheless, it can be shown that if $\Omega$ is of class $C^2$, then $|\overline{\Sigma}| = 0$;
moreover, in this case $\overline{\Sigma}\subset\Omega$ and $d_{\Omega}$ is of class $C^2$ in $\overline{\Omega}\setminus\overline{\Sigma}$ (see e.g.\ \cite{CM}).

Under the assumption that $\Omega\subset\R^n$ is a bounded open set of class $C^2$,
for every $y\in\partial \Omega$, we denote respectively by $\nu(y)$
and $T_y \Omega$
the unique outward unit normal vector
and the tangent space
of $\partial \Omega$ at $y$.
The map $\nu\colon\partial\Omega\to S^{n-1}$
is called the
\textsl{spherical image map}
(or \textsl{Gauss map}).
Then, for every $y \in \partial \Omega$, we denote by $\lambda(y) $ the \textsl{cut value} of $y$ intended according to definition (\ref{cut}).
It is well known that the singular set $\Sigma$ is a subset of
the collection $C$ of all
\textsl{cut points} $y - \lambda (y) \nu (y )$, $y\in\partial\Omega$;
this set $C$ has always vanishing Lebesgue measure
(see \cite{Bi}) and is contained in $\overline{\Sigma}$.


For every $y\in\partial \Omega$,
the differential $d\nu_y$ of the Gauss map at $y$
maps the tangent space $T_y \Omega$ into itself.
The linear map
$L_y:= d\nu_y\colon T_y \Omega\to T_y \Omega$ is called the
\textsl{Weingarten map}.
The bilinear form defined on $T_y \Omega$
by
$S_y(v,w) = \pscal{L_y\, v}{w}$,
$v,w\in T_y \Omega$,
is the
\textsl{second fundamental form} of $\partial \Omega$ at $y$.
The geometric meaning of the Weingarten map
is the following:
for every $v\in T_y \Omega$ with unit norm,
$S_y(v,v)$ is equal to the normal curvature of
$\partial \Omega$ at $y$ in the direction $v$,
that is,
$S_y(v,v) = -\langle \ddot{\xi}(0),  \nu(y) \rangle$,
where $\xi(t)$ is any parameterized curve in
$\partial \Omega$ such that $\xi(0) = y$
and $\dot{\xi}(0) = v$.
The eigenvalues $\kappa_1(y),\ldots,\kappa_{n-1}(y)$
of the Weingarten map $L_y$ are,
by definition, the
\textsl{principal curvatures} of $\partial \Omega$ at $y$.
The corresponding eigenvectors are called
the \textsl{principal directions}
of $\partial \Omega$ at $y$.
It is readily shown that every $\kappa_i(y)$ is
the normal curvature of
$\partial \Omega$ at $y$ in the direction
of the corresponding eigenvector.
{}From the $C^2$ regularity assumption on the manifold $\partial \Omega$,
it follows that the principal curvatures of $\partial \Omega$
are continuous functions on $\partial \Omega$.
Their arithmetic mean is the mean curvature $H$ of $\partial \Omega$, {\it cf.}\ definition (\ref{mc}).
The following classical identity  is usually referred as \textsl{Minkowski integral formula} for the mean curvature, see for instance \cite[Section 2A]{MoRo}:
\begin{equation}\label{f:mink}
\int_{\partial\Omega} H(y) \pscal{y}{\nu(y)} \, d\haus
= |\partial\Omega|.
\end{equation}

Finally, we shall need some elementary results about the relationship between cut value and boundary curvatures.
In any space dimension, we have the following upper bound:

\begin{lemma}\label{l:kd}
Let $\Omega\subset\R^{n}$ be a bounded connected open set of class $C^2$.
Then, for every $y\in\partial\Omega$, it holds
$$\kappa_i(y) \ll(y) \leq 1  \qquad \forall \, i=1,\ldots,n-1\ .$$
\end{lemma}

\begin{proof}
Given $y\in\partial\Omega$ let $x := y - \ll(y) \nu(y)$ be
its cut point.
Since $d_{\Omega}(x) = \ll(y)$, the open ball $B_{\ll(y)}(x)$
is contained in $\Omega$ and is tangent to $\partial\Omega$ at $y$.
Therefore, for every $i\in\{1,\ldots,n-1\}$ 
we have either $\kappa_i(y)\leq 0$ or $1 / \kappa_i(y) \geq \ll(y)$.
\end{proof}

\bigskip
In dimension $n=2$, the cut value of boundary points with maximal curvature can be easily characterized as follows:

\begin{lemma}\label{l:focal} Let $\Omega\subset\R^2$ be a bounded connected open set of class $C^2$, and let $y_0\in\partial\Omega$ be a maximum point of the curvature $\kappa$ of $\partial \Omega$. Then
 $$\lambda(y_0) = 1/\kappa(y_0)\ .$$
\end{lemma}

\begin{proof}
We observe that, since $\partial\Omega$ is compact, we have that $\kappa(y_0)>0$.
From Lemma \ref{l:kd}, we know that $\kappa(y_0)\lambda(y_0)\leq 1$.
Let us show the converse inequality. Setting $r := 1/\kappa(y_0)$,
by assumption we have that $\kappa(y)\leq 1/r$ for every $y\in\partial\Omega$.
By Schur's theorem for plane curves (see \cite[Sect.~5-7, Ex.7]{doCarmo})
we have that the disk
$B = y_0 - r \nu(y_0) + B_r(0)$ is contained in $\Omega$,
so that $\lambda(y_0)\geq r$.
\end{proof}
\bigskip

\section{The geometric result}\label{proof}

This section is devoted to prove Theorem \ref{t:geom2} and to extend it to less regular domains in the two-dimensional case. 

\bigskip
{\it Proof of Theorem \ref{t:geom2}}.
We divide the proof into four steps.

\smallskip

\textsl{Step 1: upper bound for $\varphi (y) H (y)$.}

As a first step, we claim that the following inequality holds true:
\begin{equation}\label{f:basic}
\varphi(y) H (y) \leq \frac{1}{n} \qquad \forall \, y \in \partial \Omega\ .
\end{equation}
Indeed, let
us fix $y\in\partial\Omega$ and set
$x_j := \ll(y)\kappa_j(y)$, $j=1,\ldots,n-1$.
Since $\ll(y) > 0$, using the change of variables $s = t/\ll(y)$ in the integral
in (\ref{f:phi}) and multiplying by $H(y)$ we obtain the equality
\begin{equation}\label{f:hphi}
\varphi(y) H(y) =
\ll(y) H(y)  \int_0^1 \prod_{j=1}^{n-1}(1-s x_j)\, ds =
\ff(x_1,\ldots, x_{n-1})
\,,
\end{equation}
where $\ff$ is the function defined by
\begin{equation}\label{f:f}
\ff(x_1,\ldots,x_{n-1}) := \frac{x_1+\cdots+x_{n-1}}{n-1}
\int_0^1  \prod_{j=1}^{n-1} (1-s\, x_j)\, ds\,.
\end{equation}
In view of (\ref{f:hphi}) and (\ref{f:f}), and since
$x _j \leq 1$ by Lemma~\ref{l:kd}, in order to find an upper bound for the product $\varphi(y)H(y) $, we have to maximize $f$  on the set
$$D := \{(x_1,\ldots,x_{n-1}) \in \R^{n-1}:\ x_j\leq 1\ \forall j\}\ .$$
Since $\ff(x_1,\ldots,x_{n-1}) < 0$ if $x_1+\ldots +x_{n-1} < 0$,
it is clear that $\max_D \ff$ is achieved
on the compact set
$K := D \cap \{x_1+\cdots +x_{n-1}\geq 0\}$ .
Let $(x_1, \ldots, x_{n-1}) \in K$
and denote by $\overline{x}$ their arithmetic mean. From the arithmetic--geometric mean inequality, taking into account that $\overline x \geq 0$,
we have
\[
\begin{array}{ll} \displaystyle{\ff(x_1,\ldots,x_{n-1}) }& \displaystyle{= \overline{x}\int_0^1 \prod_{j=1}^{n-1} (1-s\, x_j)\, ds}\\ &\displaystyle{ \leq
\overline{x}\int_0^1 \left(1-s\overline{x}\right)^{n-1} ds }\\ \noalign{\medskip} &  \displaystyle{= \frac{1}{n} \big [ 1 - ( 1 - \overline x) ^ n \big ] \leq \frac{1}{n}}\ . \end{array}
\]
Then the inequality (\ref{f:basic}) is proved. 

\smallskip
\textsl{Step 2: upper bound for $H (y)$. }

Exploiting
the assumptions (i) and (ii) made on the point $y _0$, and applying (\ref{f:basic}) at $y = y _0$, we obtain
\begin{equation}\label{f:dis1}
\frac{|\Omega|}{|\partial\Omega|}\, H(y) \leq
\frac{|\Omega|}{|\partial\Omega|}\, H(y_0) \leq
\varphi(y_0) H(y_0) \leq \frac{1}{n}
\qquad\forall y\in\partial\Omega\,.
\end{equation}

\smallskip
\textsl{Step 3: conclusion.}

We multiply (\ref{f:dis1}) by the positive quantity
$\pscal{y}{\nu(y)}$: integrating on $\partial\Omega$,
thanks to Minkowski integral formula (\ref{f:mink}) and the divergence theorem, we get
\[
|\Omega| = \frac{|\Omega|}{|\partial\Omega|}
\int_{\partial\Omega} H(y) \pscal{y}{\nu(y)} \, d\haus
\leq \frac{1}{n}\int_{\partial\Omega} \pscal{y}{\nu(y)} \, d\haus
= |\Omega|\,.
\]
We deduce that
$$\frac{|\Omega|}{|\partial\Omega|}\, H(y) \pscal{y}{\nu(y)} = \frac{1}{n} \pscal{y}{\nu(y)}\qquad \forall \, y \in \partial \Omega \ .$$
Recalling again that that $\pscal{y}{\nu(y)}$ is a (strictly) positive quantity, we infer that the mean curvature of $\partial\Omega$ is constant,
hence $\Omega$ is a ball by Alexandrov's Theorem \cite{Al}
(see also \cite{Ros}).

\smallskip

\textsl{Step 4: the case when $\varphi$ is constant.}

In order to obtain the second part of the statement, we recall the following change of variable's formula
which has been proved in \cite[Theorem 7.1]{CM} for every $h \in L ^ 1 (\Omega)$:
\begin{equation}\label{chv}
\int_{\Omega} h(x)\, dx =
\int_{\partial\Omega}\left[
\int_0^{\ll(y)}
h(y - t \nu(y))\,
\prod_{j=1}^{n-1}[1-t \kappa_j(y)]\, dt
\right]\, d\haus(y) \, .
\end{equation}
By applying it with $h \equiv 1$, we obtain that the mean value of $\varphi$ over $\partial \Omega$ is precisely $|\Omega| / |\partial \Omega|$.
Hence, if $\varphi$ is constant on $\partial \Omega$, both conditions in (\ref{f:hypoa}) are fulfilled by choosing as $y _0$ any point of maximal mean curvature.
\qed

\bigskip
In the remaining of this section, we assume without any further mention that $n=2$, 
and we study the validity of Theorem \ref{t:geom2} for domains with piecewise smooth boundary.  We start by noticing that, if ``concave" corners are present in $\partial \Omega$, the change of variables
formula \eqref{chv} is no longer valid. In fact, in presence of such kind of corners, assuming that the function $\varphi$ is constant in their complement, 
the conclusion of Theorem~\ref{t:geom2} does not remain true, as shown by the following

\begin{example}
Consider the set 
$\Omega = B_2(-1,0) \cup B_2(1,0) \subset\R^2$.
In this case $\Omega$ has two ``concave" corners at
$\mathcal{C} = \{(0,\sqrt{3}), (0, -\sqrt{3})\}$,
$\varphi(y) = 1$ for every $y\in\partial\Omega\setminus\mathcal{C}$,
but $\Omega$ is not a ball.
\end{example}

In view of the above Example, we are going to restrict our attention to the following class of domains.

\begin{definition}
\label{d:ps}
We say that a bounded connected open set $\Omega\subset\R^2$
is \textsl{piecewise $C^2$ without concave corners} if
it satisfies a uniform
exterior sphere condition, and
\[ \partial\Omega = \bigcup_{i=1}^m \Gamma_i\ ,\] 
where each
$\Gamma_i$ is a simple $C^2$ arc up to its endpoints
$y_{i-1}$ and $y_{i}$ (with the convention $y_0 = y_m$), and 
\[\Gamma_i\cap\Gamma_j =
\begin{cases}
\{y_i\}, & \text{if}\ 1\leq i\leq m-1,\ j = i+1,\\
\{y_m\}, & \text{if}\ i=m,\ j=1,\\
\emptyset, & \text{if}\ i-j\neq \pm 1\ .
\end{cases}
\]
\end{definition}

Clearly, the boundary of a domain $\Omega$ as in Definition~\ref{d:ps} 
may contain ``convex'' corners at points in the set
$${\mathcal C} := \{ y _1, \dots, y _m \}\ .$$
We point out that, for such domains, 
the change of variables formula \eqref{chv} still holds
(see \cite[Thm.~3.3]{CFV}), whereas
the Minkowski integral formula \eqref{f:mink} becomes
\begin{equation}
\label{f:minkm}
|\partial\Omega| = \int_{\partial\Omega} \kappa(y)
\pscal{y}{\nu(y)}\, d\mathcal{H}^1(y)
- \sum_{i=1}^m \pscal{y_i}{R \Delta\nu(y_i)}\,,
\end{equation}
where $R:= \begin{pmatrix}
0 & -1\\ 1 & 0\end{pmatrix}$, $\Delta\nu(y_i) = \nu_{i+1}(y_i) - \nu_i(y_i)$
(being $\nu_i$ the outer normal to the boundary component $\Gamma_i$),  and 
\[
\nu_i(y_j) := \mathop{\lim_{y\to y_j}}_{y\in\Gamma_i} \nu_i(y)
\]
(see \cite{ABF}).

%

Consequently, the following generalized versions of Theorem \ref{t:geom2}   hold. 

\begin{theorem}\label{t:geom2b}
Let $\Omega\subset\R^2$ be a 
bounded connected open set
piecewise $C^2$ without concave corners, and 
starshaped with respect to the origin.
Assume that one of the following conditions is satisfied:
\begin{itemize}
\item[(i)] $\Omega$ is of class $C^1$ and
there exists a point $y_0\in\partial\Omega\setminus\mathcal{C}$
such that
\[
H(y_0) = \max_{y\in\partial\Omega\setminus\mathcal{C}} H(y)\,,\qquad
\varphi(y_0) \geq \frac{|\Omega|}{|\partial\Omega|}
\,.
\]
\item[(ii)] 
$\varphi$ is constant in
$\partial\Omega\setminus\mathcal{C}$.
\end{itemize}
Then $\Omega$ is a ball. 
\end{theorem}

\begin{proof} Assume first that $\Omega$ satisfies condition (i). 
In this case, all the steps of the proof of Theorem~\ref{t:geom2}
can be repeated. In particular, Step 3 continues to hold thanks to the hypothesis that that 
$\Omega$ is of class $C^1$: indeed, such assumption ensures that
each addendum of the finite sum in the r.h.s.\ of
(\ref{f:minkm}) vanishes, so that  the Minkowski formula is still valid for $\Omega$ in the very same form as if it was a $C^2$ domain. 

Assume now that $\Omega$ satisfies condition (ii). 
Similarly as in Step 4 of the proof of Theorem~\ref{t:geom2}, since $\varphi$ is constant, 
from the
change of variable's formula \eqref{chv} we 
get that $\varphi(y) = |\Omega| / |\partial\Omega|$
for every $y\in\partial\Omega\setminus\mathcal{C}$.

This fact has two consequences. 

As a first consequence,  the upper bound inequality \eqref{f:dis1} in Step 2
holds in the form
\begin{equation}\label{f:dis1b}
\frac{|\Omega|}{|\partial\Omega|}\, \kappa(y)\leq \frac{1}{2}\,,
\qquad\forall y\in\partial\Omega\setminus\mathcal{C}.
\end{equation}
Namely, if $\kappa(y) \leq 0$ the inequality is trivial,
whereas, if $\kappa(y) > 0$, then $\lambda(y) \leq 1/\kappa(y)$ and
\[
\frac{|\Omega|}{|\partial\Omega|}
= \varphi(y) =
\lambda(y) - \frac{\lambda(y)^2}{2}\, \kappa(y)
\leq \frac{1}{2\kappa(y)}
\]
and \eqref{f:dis1b} follows.

As a second consequence,  $\partial\Omega$ cannot have (convex) corners.
Namely, if $y_i$ is a point with $\pscal{\nu_i(y_i)}{\nu_{i+1}(y_i)} < 1$,
a simple geometric argument shows that $\lim_{y\to y_i} \ll(y) = 0$.
Since $\kappa$ is uniformly bounded on 
$\partial\Omega\setminus\mathcal{C}$, 
we deduce that $\lim_{y\to y_i} \varphi(y) = 0$
and,
by assumption, we conclude that $\varphi = 0$ on
$\partial\Omega\setminus\mathcal{C}$, 
a contradiction.

Now, since \eqref{f:dis1b}  holds and  the absence of convex corners ensures that the Minkowski formula is satisfied for $\Omega$ in the very same form as if it was a $C^2$ domain, the conclusion of the proof 
is identical to Step 3 in the proof of Theorem~\ref{t:geom2}.
\end{proof}

\section{Application to PDE's of Monge-Kantorovich type}\label{secsabbia}

In this section we deal with question (\ref{Q2}) stated in the Introduction.
We begin by explaining more in detail its relationship with the overdetermined problem (\ref{deqp}) for the $p$-Laplacian.
Let $\Omega\subset \R ^n$ be an open bounded connected set of class $C^2$,  and let $f$ be a signed measure with finite total variation supported in $\overline \Omega$.  For every $p >1$, set
$$\alpha _p := \min \Big \{ \int _{\Omega} \frac{1}{p} |\nabla u | ^ p \, dx - \langle f , u  \rangle \ :\ u \in W ^ {1,p} _0 (\Omega) \Big \}\ ,$$
and denote by $u _p$ the unique minimizer, namely the unique solution to  the boundary value problem
\[\left\{\begin{array}{rll}
- \Delta_p u =f\qquad&\hbox{ in }\Omega\,,\\
u=0\qquad&\hbox{ on }\partial\Omega\, .
\end{array}\right.
\]
It was proved in \cite{BoBuDe} that, as $p \to + \infty$,
\begin{eqnarray}
&\alpha _p \to \alpha & \label{conv0} \\
\noalign{\smallskip}
& u _p \to u \ \ \hbox{ uniformly } & \label{conv1} \\
\noalign{\smallskip}
& |\nabla u _p| ^ { p-2} \rightharpoonup \mu \ \ \hbox{weakly as measures\ ,} & \label{conv2}
\end{eqnarray}
being
\begin{equation}\label{alpha}\alpha:= - \sup \Big \{ \langle f, u \rangle \ : \ u \in C^ \infty  _ 0 (\R^n) \,, \ |\nabla u| \leq 1 \hbox{ on } \Omega\, , \ u = 0 \hbox{ on } \partial \Omega \Big \}
\end{equation}
and
$(u, \mu)$  a solution to
\begin{equation}\label{deqp23}
\begin{cases}
- {\rm div} ( (D _\mu u) \mu) =f &\text{in}\ \Omega\,,\\
u \in {\rm Lip }_1 (\Omega, \partial \Omega)\,, \\
| D_\mu u| = 1 & \mu\hbox{-a.e.}
\end{cases}
\end{equation}
Here
\[
{\rm Lip } _1 (\Omega, \partial \Omega) = \Big \{ u \in W ^ {1, \infty} (\Omega)\ :\ u = 0 \hbox{ on } \partial \Omega \ \hbox{ and } \ |\nabla u| \leq 1 \hbox{ a.e. in } \Omega \Big \}\ ,
\]
where, thanks to the regularity of $\Omega$, we understand that a function
$u\in W^{1,\infty}(\Omega)$ is continuously extended to $\overline{\Omega}$,
while $\mu$ is a positive measure supported on $\overline \Omega$,  and $D_\mu u$ denotes the  $\mu$-tangential gradient of $u$, which can be suitably defined
starting from the notion of tangent space to $\mu$ introduced in \cite{BoBuSe} (see also \cite{FrMa}).

Here we skip to introduce the tangential calculus with respect to a measure (for which we refer {\it e.g.} to the survey paper \cite{BoFr}),
since we are going to restrict our attention to the case when $\mu$ is an absolutely continuous measure:
$$
\mu = v (x) \, dx \, , \hbox{ with }   v \in L ^ 1 (\Omega), \ \ v \geq 0\ .
$$
In this case, the $\mu$-tangential gradient of $u$ agrees with the gradient of $u$ on the set $\{ v>0 \}$, and vanishes on the complementary set $\{ v = 0 \}$. Thus system (\ref{deqp23}) can be rewritten as

\begin{equation}\label{deqp3}
\begin{cases}
- {\rm div} ( v \nabla u) =f &\text{in}\ \Omega\,,\\
u \in {\rm Lip }_1 (\Omega, \partial \Omega)\,, \\
( 1 - |\nabla u| ) v = 0  & \text{a.e.\ in}\ \Omega\, .
\end{cases}
\end{equation}

In view of  the convergence property (\ref{conv2}), we paraphrase the classical overdetermined Neumann condition for system (\ref{deqp}), namely
$$|\nabla u _p | = c \qquad \hbox{ on } \partial \Omega\ ,$$
with the the following overdetermined condition for system (\ref{deqp23}):
\begin{equation}\label{muac2}
\mu = v (x) \, dx \, , \hbox{ with }  v \in L ^ 1 (\Omega), \ \ v \geq 0,\  \ v= c \hbox{ on } \partial \Omega\ .
\end{equation}
By this way, we arrive at question (\ref{Q2}): under the assumption that the source $f$ is a positive constant on $\overline \Omega$, does the existence of a solution to (\ref{deqp23})-(\ref{muac2}) imply the roundness of $\Omega$?

\medskip

\begin{remark}{\rm
We point out that the constancy condition asked for $v$ on the boundary in (\ref{muac2}) is meaningful as soon as the source $f$ is assumed to be a continuous function on $\overline \Omega$, because in this case the $v$-component of a solution $(u, v)$ to (\ref{deqp3})  is itself a continuous function
(see \cite[Prop.~3.2]{CCCG}).
Moreover, we recall that the $v$-component is unique
(see \cite[Thm.~4.1]{CCCG}, \cite[Section 6]{CM10} and the proof
of Theorem~\ref{teosabbia} below),
whereas the $u$-component is unique (and coincides with $d_{\Omega}$)
if and only if the singular set $\Sigma$ of $d _\Omega$ is contained in the support of
the source $f$ \cite[Section 7]{CM10}.
In particular, if $f$ is a positive constant, then $u=d_{\Omega}$
and $u_p$ converges uniformly to $d_{\Omega}$ as $p\to\infty$
(see also \cite{BDM} and \cite[Remark~2.1]{BuKa}).}
\end{remark}

\medskip
 The next result gives an affirmative answer to question (\ref{Q2}) (under the appropriate assumptions on $\Omega$ which allow to apply Theorem \ref{t:geom2}; a similar statement clearly holds in 
the nonsmooth two-dimensional case covered by Theorem~\ref{t:geom2b}).

\begin{theorem}\label{teosabbia}
Let $\Omega\subset\R^n$ be a bounded connected open set of class $C^2$,
starshaped with respect to the origin.
If the source $f$ is positive and constant on $\overline \Omega$ and problem $(\ref{deqp23})$-$(\ref{muac2})$ admits a solution, then $\Omega$ is a ball.
\end{theorem}

\begin{proof}
Since $f\in C(\overline{\Omega})$,
one solution to problem (\ref{deqp3}) is given
by the pair $(d_{\Omega}, v_f)$, where $v_f\colon\overline{\Omega}\to \R$ is the
(continuous) function defined by
\begin{equation}\label{f:vf}
v_f(x) =
\begin{cases}
\displaystyle{
\int_0^{\tau(x)}
f(x+t\nu(x))
\prod_{i=1}^{n-1}\frac{1-(d_{\Omega}(x)+t)\, \kappa_i(x)}%
{1-d_{\Omega}(x)\, \kappa_i(x)}\, dt}
&\textrm{if $x\in{\overline{\Omega}}\setminus\overline{\Sigma}$},\\
0,
&\textrm{if $x\in\overline{\Sigma}$}
\end{cases}
\end{equation}
(see \cite[Thm.~3.1]{CCCG}).
Here, for every $x\in\overline{\Omega}\setminus\overline{\Sigma}$,
we have used the notation
$\kappa_i(x) := \kappa_i(\pi(x))$, $\nu(x) := \nu(\pi(x))$,
and $\tau(x) := \ll(\pi(x)) - d_{\Omega}(x)$.
Moreover, the following uniqueness result holds for
the $v$-component of the system:
if $(u,v)$ is any solution to (\ref{deqp3}), then
$v = v_f$
(see \cite[Thm.~4.1]{CCCG}).

In particular, in the case $x\in\partial\Omega$ (i.e.\ $d_{\Omega}(x) = 0$), the representation formula (\ref{f:vf}) can be rewritten as:
\[v_f(x) =
\int_0^{\ll(x)}
f(x+t\nu(x))
\prod_{i=1}^{n-1}(1-t \kappa_i(x))\, dt,
\qquad x\in\partial\Omega.
\]
From such representation formula at the boundary and the uniqueness of the
$v$-component, when $f$ is a positive constant $\gamma$  for every $x\in\overline{\Omega}$,
 we have
\[
v (y) = v_f(y) = \gamma \varphi(y)\qquad \forall y\in\partial\Omega,
\]
where $\varphi$ is precisely the function defined at (\ref{f:phi}).
The conclusion now follows from Theorem~\ref{t:geom2}.
\end{proof}

\bigskip
Theorem \ref{teosabbia} admits the following physical interpretations.

\bigskip
-- {\it A thermic model}.  Consider the problem of finding  a positive measure $\mu$ which solves the variational problem
\begin{equation}\label{pbBB}
\min \Big \{ {\mathcal C} (\mu) \ :\ \int d \mu = m \,,\ {\rm spt} (\mu) \subseteq \overline \Omega \Big \}\ ,
\end{equation}
where $m$ is a prescribed positive parameter, and the cost ${\mathcal C}$ is given by
$${\mathcal C} (\mu) :=  -\inf \Big \{ \int  j (\nabla u) \, d \mu - \langle f, u \rangle  \ :\ u \in {\mathcal D} (\R ^n) \ ,\ u = 0 \hbox{ on } \partial \Omega \Big \} \ ,$$
being $j$ a stored energy density (typically, $j (z) = \frac{1}{2}|z| ^ 2$), and $f$ a signed measure with finite total variation supported on $\overline \Omega$. If
$\mu$ is interpreted as the distribution of a conducting material, $u$ as the temperature (which is kept $0$ at the boundary), and
$f$ as a given heat sources density, (\ref{pbBB}) can be interpreted as the optimal design problem of finding the most performant conductor of prescribed mass in the design region $\overline \Omega$.

A key result established in \cite[Theorem 2.3]{BoBu} states that the minimum in (\ref{pbBB}) is equal to
$$\frac{\alpha ^ 2}{2m}\ ,$$
with $\alpha$ as in (\ref{conv0}).
Moreover, up to constant multiples, system (\ref{deqp23}) provides necessary and sufficient optimality conditions on $u$ and $\mu$ for being a solution respectively to (\ref{alpha}) and (\ref{pbBB}) \cite[Theorem 3.9]{BoBu}.

In the light of these results, Theorem \ref{teosabbia} can be interpreted as follows:

\medskip
{\it Assume that a constant source heats a region $\overline \Omega$, and that the most performant conductor in such region is represented by a function constant on the boundary $\partial \Omega$. Then $\Omega$ is a ball.}

\bigskip
-- {\it A model for sandpiles}.
In the dynamical theory of granular matter the so-called table problem
consists in studying the evolution of a sandpile created by pouring
dry matter (the sand) onto a table, which is represented by a bounded open
set $\Omega\subset\mathbb{R}^2$, while the (time-independent) matter
source is represented by a non-negative function $f\in L^1(\Omega)$.
Among differential models, in the so-called BCRE model
(after Bouchaud, Cates, Ravi Prakash, Edwards \cite{BCRE}) and its successive modifications
(see \cite{HK}), the description of the growing sandpile
is based on the introduction of two layers,
the \textsl{standing layer}, which collects the matter
that remains at rest, and the
\textsl{rolling layer}, which is the thin layer of matter
moving down along the surface of the standing layer.

The equilibrium solutions of this model
are precisely the solutions to (\ref{deqp3}),
where $u$ and $v$ denote respectively the thickness of
the standing and rolling layer.

Theorem \ref{teosabbia} has then the following straightforward
meaning:

\medskip
{\it Assume that a uniformly distributed (and constant in time)
amount of dry sand is poured onto a table $\Omega\subset\mathbb{R}^2$,
and that, once the equilibrium is reached, the
thickness of the rolling layer is constant
on the boundary $\partial \Omega$. Then $\Omega$ is a disk.}


\bigskip

\section{Application to PDE's with partially web solutions}\label{secparweb}

In this section we deal with question (\ref{Q1}) stated in the Introduction. We adopt the same notation for the set
$\Omega_\Gamma$ and the space  $\W (\Omega; \Omega_\Gamma)$.

The next result gives sufficient conditions on $\Gamma$ for an affirmative answer in space dimensions $n=2$;
$\kappa$ denotes the curvature of $\partial \Omega$.

\begin{theorem}\label{teopartialweb} Let $\Omega\subset\R^2$ be a bounded connected open set of class $C^2$,
starshaped with respect to the origin.
Assume that $A\in C([0,+\infty))$ and that
there exists a solution $u$ to equation \eqref{f:elleq} in $\Omega$ 
belonging to the space
\begin{equation}\label{f:w1}
\W^1 (\Omega; \Omega_\Gamma) :=
\left\{
u\in \W (\Omega; \Omega_\Gamma)\ : \
u|_{\overline{\Omega_\Gamma}} \in C^1
\right\}\,,
\end{equation}
where $\Gamma$ is a relatively open connected subset of $\partial \Omega$
such that
\begin{itemize}
\item[(i)] the maximum of $\kappa$  on $\partial \Omega$ is attained on $\Gamma$;
\item[(ii)] for every $\epsilon > 0$ small enough,
a.e.\ in
$\Omega_{\epsilon} := \{x\in\Omega:\ d_{\Omega}(x) < \epsilon\}$
there holds:
\[
A(|\nabla u|)\pscal{\nabla u}{\nabla d_{\Omega}} \leq
 [-A(|\nabla u|) u_{\nu}] {\big | _ \Gamma} + o(1)
\]
(where $o(1)$ denotes a quantity that vanishes as $\epsilon\to 0$).
\end{itemize}
Then $\Omega$ is a ball.
\end{theorem}

\begin{remark}{\rm In assumption (ii),  $[-A(|\nabla u|) u_{\nu}] {\big | _ \Gamma}$ denotes the restriction of $[-A(|\nabla u|) u_{\nu}]$ to $\Gamma$, which is constant since $u \in \W (\Omega; \Omega_\Gamma)$.
We point out that, if one
assumes further that $u$ is of class $C^1(\overline{\Omega_\epsilon})$,
condition (ii) can be rewritten in a more comfortable way. Namely, in this case it becomes
\begin{itemize}
\item[(ii')]
the maximum of $( - A (|\nabla u|) u _\nu)$ on $\partial \Omega$ is attained on $\Gamma$.
\end{itemize}}
\end{remark}

\medskip
Theorem \ref{teopartialweb}  should be compared with the results obtained for partially overdetermined boundary value problems obtained in \cite{FG}.
Therein, it was proved in particular that existence of a solution to the boundary value problem
\begin{equation}\label{deqtris} \left\{\begin{array}{rll}
- \Delta u=1\qquad&\hbox{ in }\Omega\,,\\
u=0\qquad&\hbox{ on }\partial \Omega\,, \\
|\nabla u| =c\qquad&\hbox{ on }\Gamma\,,
\end{array}\right.
\end{equation}
implies that $\Omega$ is a ball under one of the following assumptions:

\medskip
{--} $\partial \Omega$ is connected, and $\Gamma \subseteq \partial \widetilde \Omega$ for some open set $\widetilde \Omega$ with connected analytic boundary
({\it cf.} \cite[Theorem 1]{FG})
;

\medskip
{--} $\partial \Omega \in C ^ {2, \alpha}$, $\sup _{x \in \Gamma} H (x) \geq {1/nc}$ and the maximum of $|\nabla u|$ over $\partial \Omega$ is attained on $\Gamma$ ({\it cf.} \cite[Theorem 3]{FG}).

\medskip
In comparison with \cite[Theorem 1]{FG}, let us remark that the proof of Theorem \ref{teopartialweb} is straightforward in case $A$ is the Laplacian
or any other operator in divergence form for which one knows the following two facts: equation $(\ref{f:elleq})$ has a analytic solution inside $\Omega$ and  Serrin's symmetry result holds true. Indeed in this case, by uniqueness of the analytic continuation, the solution which is assumed to exist in ${\mathcal W} (\Omega; \Omega _\Gamma)$ agrees with the analytic solution to $(\ref{f:elleq})$ on whole of $\Omega$, so that $u$ satisfies at the same time a constant Dirichlet and Neumann condition on the entire $\partial \Omega$, and the conclusion follows from Serrin's result.
Clearly, this kind of argument does not apply any longer when dealing with degenerated elliptic operators such as the $p$-Laplacian, for which analytic regularity of solutions on the whole of $\Omega$ is not fulfilled (see {\it e.g.} \cite{DiBe, Tolk}).

On the other hand, it is interesting to observe that the assumptions of Theorem \ref{teopartialweb} are quite similar to those of \cite[Theorem 3]{FG}, though (as already mentioned in the Introduction) neither the existence of a solution $u \in \W (\Omega; \Omega _\Gamma)$ implies the existence of a solution to (\ref{deqtris}) nor the converse, and though the proof techniques are completely different.

The proof of Theorem \ref{teopartialweb} relies on Theorem \ref{t:geom2} combined with Proposition \ref{teo10} below, where we establish a
link holding, at a point of maximal curvature, between the normal derivative of partially web solutions to $(\ref{f:elleq})$,
and the function $\varphi$ introduced in (\ref{f:phi}).

\begin{proposition}\label{teo10}
Let $\Omega\subset\R^2$ be a bounded connected open set of class $C^2$,
starshaped with respect to the origin.
Assume there exists a solution $u$ to equation $(\ref{f:elleq})$ in $\Omega$ belonging to the space $\W^1 (\Omega; \Omega_\Gamma)$
defined in \eqref{f:w1},
where $\Gamma$ is a relatively open connected subset of $\partial \Omega$ such that
the maximum of $\kappa$  on $\partial \Omega$ is attained at some point  $y _0 \in\Gamma$.

Then the following identity holds at $y _0$:
\[
A(|\nabla u | (y _0)) u _\nu (y _0)   \displaystyle{=  -\varphi(y_0)   } \ .
\]
\end{proposition}
\proof

Let $Y\colon (-r,r)\to\partial\Omega$
be a local parametrization of $\Gamma$ by arc-length,
such that $Y(0) = y_0$. For $\rho\in (-r,r)$, let
$$\Gamma _\rho:= Y ( -\rho, \rho)\ .$$
For $\rho$ as above and $t \in [0, \Lambda (\rho)]$,
we set
$$\Lambda(\rho) := \lambda(Y(\rho))\,, \qquad
K(\rho) := \kappa(Y(\rho))\ ,$$
and
$$T(\rho) := \tau(Y(\rho)) = \dot Y (\rho)\,, \qquad  N(\rho) := \nu(Y(\rho))\, ,  \qquad X (\rho, t) = Y (\rho) - t N (\rho)\ .$$
By assumption, \eqref{f:elleq}  admits a solution 
$u \in \W^1 (\Omega; \Omega_\Gamma)$,
so we can write
$u(x) = h(d_{\Omega}(x))$ for every $x\in\overline{\Omega_{\Gamma}}$, with $h \in C^ 1$, and it holds
\begin{equation}\label{gradu}\nabla u ( X (\rho , t)) = - h' (t) N (\rho)\ .\end{equation}
We now construct a suitable family of test function to be used in equation $(\ref{f:elleq})$.

Let $\Lambda_{\rho} := \min_{|\sigma|\leq\rho} \Lambda(\sigma)$.
For $\epsilon>0$ small enough let
$\phi_{\epsilon}, \eta_{\epsilon}\colon\R\to\R$
be the functions defined by
\[
\phi_{\epsilon}(t) :=
\begin{cases}
0 &
\textrm{if $t\leq 0$ or $t\geq \Lambda_{\rho} - \epsilon$,}\\
1 &
\textrm{if $t\in [\epsilon, \Lambda_{\rho}-2\epsilon]$,}\\
\frac{t}{\epsilon} &
\textrm{if $t\in (0, \epsilon)$,}\\
\frac{\Lambda_{\rho} - \epsilon-t}{\epsilon} &
\textrm{if $t\in (\Lambda_{\rho}-2\epsilon, \Lambda_{\rho} - \epsilon)$,}
\end{cases}\quad
\eta_{\epsilon}(\sigma) :=
\begin{cases}
0 & \textrm{if $|\sigma|\geq \rho$,}\\
1 & \textrm{if $|\sigma|\leq \rho-\epsilon$,}\\
\frac{\rho-|\sigma|}{\epsilon} &
\textrm{if $\rho-\epsilon < |\sigma|< \rho$.}
\end{cases}
\]
Then, for $\rho \in (- r , r)$ and $\epsilon $ small enough, we consider the family of functions  $\psi_{\rho,\epsilon}\colon\Omega\to\R$ given by
\[
\psi_{\rho,\epsilon}(x):=
\begin{cases}
\phi_{\epsilon}(t) \eta_{\epsilon}(\sigma),
&\text{if $x = X(\sigma, t)$ for some $(\sigma, t)\in D_{\rho} := (-\rho,\rho)\times
(0, \Lambda_{\rho}-\epsilon)$},\\
0,
&\text{otherwise}.
\end{cases}
\]
Since $X$ is a $C^1$ diffeomorphism from $\overline{D_{\rho}}$ to $X(\overline{D_{\rho}})$,
each function $\psi_{\rho,\epsilon}$ is Lipschitz continuous.
Therefore, it can be taken as a test function in equation $(\ref{f:elleq})$.
Passing to the limit as $\epsilon \to 0$, then dividing by $|\Gamma _\rho |$
and passing to the limit also as $\rho \to 0$,
we obtain
\begin{equation}\label{f:eq1}
\lim _{\rho \to 0} \Big \{ \frac{1}{|\Gamma _\rho |} \lim _{\epsilon \to 0} \int _\Omega \langle A (|\nabla u|) \nabla u, \nabla \psi _{\rho, \epsilon } \rangle  \, dx
\Big \} = \lim _{\rho \to 0} \Big \{ \frac{1}{|\Gamma _\rho |} \lim _{\epsilon \to 0} \int _\Omega \psi _{\rho, \epsilon } \, dx\Big \}\ .
\end{equation}
The right hand side of (\ref{f:eq1}) is immediately computed as
\begin{equation}\label{f:eq2}
\lim _{\rho \to 0} \Big \{ \frac{1}{|\Gamma _\rho |} \lim _{\epsilon \to 0} \int _\Omega \psi _{\rho, \epsilon } \, dx\Big \}
= \lim _{\rho \to 0} \frac{|\Omega _\rho|}{|\Gamma _\rho|} = \lim _{\rho \to 0} \frac{\int_{\Gamma_\rho} \varphi \, d {\mathcal H} ^1}{|\Gamma _\rho|}
= \varphi (y _0)\ .
\end{equation}
Let us compute the left hand side of (\ref{f:eq1}).

Differentiating the relation $\psi_{\rho,\epsilon}(X(\sigma,t)) = \phi_{\epsilon}(t) \eta_{\epsilon}(\sigma)$
with respect to $t$ we get
\begin{equation}\label{f:diffpsi}
\pscal{\nabla\psi_{\rho,\epsilon}(X(\sigma,t))}{N(\sigma)} = -\phi_{\epsilon}'(t) \eta_{\epsilon}(\sigma)
\qquad\text{a.e.\ on } D_{\rho}.
\end{equation}
By combining (\ref{gradu}) and (\ref{f:diffpsi}), we infer
\[
\begin{array}{ll} \pscal{A (|\nabla u|) \nabla u}{\nabla\psi_{\rho,\epsilon}}(X(\sigma,t)) & =
-A(|h'(t)|) h'(t) \pscal{N(\sigma)}{\nabla\psi_{\rho,\epsilon}(X(\sigma,t))} \\ \noalign{\medskip}
& = A(|h'(t)|) h'(t) \phi_{\epsilon}'(t) \eta_{\epsilon}(\sigma) \qquad \text{a.e.\ on } D_{\rho}.\\
\end{array}
\]
Then, by using the change of variables formula (\ref{chv}), we get
\[
\begin{split}
\int _\Omega  \langle A (|\nabla u|) & \nabla u, \nabla \psi _{\rho, \epsilon } \rangle  \, dx
\\ = {} &
\int _{D_\rho} \langle A (|\nabla u|) \nabla u, \nabla \psi _{\rho, \epsilon } \rangle  \, dx
\\ = {} &
\int_{-\rho}^{\rho}
\int_0 ^ {\Lambda (\sigma) }
 A(|h'(t)|) h'(t) \phi_{\epsilon}'(t) \eta_{\epsilon}(\sigma)   \big [ 1 - t K (\sigma) \big ] \, dt \, 
d\sigma
\\ = {} &
 \frac{1}{\epsilon} 
\int_{-\rho}^{\rho}
\int_0 ^ {\epsilon }
 A(|h'(t)|) h'(t)  \eta_{\epsilon}(\sigma)   \big [ 1 - t K (\sigma) \big ] \, dt \, 
d\sigma 
\\ & -
\frac{1}{\epsilon} 
\int_{-\rho}^{\rho}
\int_{\Lambda _\rho - 2 \epsilon} ^ {\Lambda _ \rho - \epsilon}
 A(|h'(t)|) h'(t)  \eta_{\epsilon}(\sigma)   \big [ 1 - t K (\sigma) \big ] \, dt \, 
d\sigma
\ .
\end{split}
\]
In the limit as $\epsilon \to 0^+$, by the regularity assumption made on the solution $u$ and on the operator $A$ (recall that 
$h\in C^1$
and $A \in C ([ 0 , + \infty))$), we obtain
\[
\begin{split}
\lim _{\epsilon  \to 0 ^+} & \int _\Omega \langle A (|\nabla u|) \nabla u, \nabla \psi _{\rho, \epsilon } \rangle  \, dx
\\ & =
\int_{-\rho}^{\rho}
\Big \{ A(|h'(0)|) h'(0) -  A(|h'(\Lambda _\rho)|) h'(\Lambda _\rho) \big [ 1 - \Lambda _\rho K (\Lambda _\rho) \big ] \Big \}
d\sigma
\\ & =
|\Gamma _\rho| \Big \{ A(|h'(0)|) h'(0) -  A(|h'(\Lambda _\rho)|) h'(\Lambda _\rho) \big [ 1 - \Lambda _\rho K (\Lambda _\rho) \big ] \Big \} \, .
\end{split}
\]
In the limit as $\rho  \to 0 ^+$, by exploiting again the regularity assumptions recalled above, we obtain
$$\lim _{\rho \to 0 ^+} A(|h'(\Lambda _\rho)|) h'(\Lambda _\rho) \big [ 1 - \Lambda _\rho K (\Lambda _\rho) \big ]   = A(|h'(\lambda (y _0))|) h'(\lambda (y _0)) \big [ 1 - \lambda (y _0) \kappa (y _0)  \big ] = 0\ ,$$
where the last equality follows from Lemma \ref{l:focal}.

We have thus proved that the left hand side of (\ref{f:eq1}) is given by
\begin{equation}\label{f:eq3}
\lim _{\rho \to 0} \Big \{ \frac{1}{|\Gamma _\rho |} \lim _{\epsilon \to 0} \int _\Omega \langle A (|\nabla u|) \nabla u, \nabla \psi _{\rho, \epsilon } \rangle  \, dx
\Big \} = A(|h'(0)|) h'(0) = - A (|\nabla u| ( y _0) ) u _\nu (y _0)\ .
\end{equation}
The proof is achieved by combining (\ref{f:eq1}),  (\ref{f:eq2}), and (\ref{f:eq3}).
\qed

\bigskip
\bigskip


\medskip
{\it Proof of Theorem \ref{teopartialweb}}. By assumption (i), there exists a point $y _0 \in \Gamma$ such that
\begin{equation}\label{f:h1b}
\kappa (y _0) = \max _{y \in \partial \Omega} \kappa (y)\ .
\end{equation}
By Proposition \ref{teo10}, it holds
\[\varphi ( y _0) = - A \big ( |\nabla u| (y _0) \big ) u _\nu ( y _0) \ .
\]
We now exploit again the assumption that $u = h (d_\Omega)$ in 
$\overline{\Omega _\Gamma}$, for some function $h:\R ^+ \to \R$
of class $C^1$.
Thus
$A \big ( |\nabla u| \big ) u _\nu $ is constant on $\Gamma$,
so that
\[
\varphi(y_0) = - A \big ( |h' (0) |)  h' (0) = [-A(|\nabla u|) u_{\nu}] {\big | _ \Gamma} 
\]
and assumption (ii) becomes
\[
A(|\nabla u|)\pscal{\nabla u}{\nabla d_{\Omega}} \leq
\varphi(y_0) + o(1) \qquad  \text{a.e.\ in}\
\Omega_{\epsilon}\,.
\]
Using $\psi_{\epsilon}(x) := \min\{ \epsilon^{-1} d_{\Omega}(x), 1\}$
as a test function in (\ref{f:elleq}) we obtain,
for $\epsilon > 0$ small enough,
\[
\int_{\Omega}\psi_{\epsilon}\, dx =
\frac{1}{\epsilon}\int_{\Omega_{\epsilon}}
A(|\nabla u|)\pscal{\nabla u}{\nabla d_{\Omega}}\, dx
\leq \varphi(y_0) |\partial\Omega| + o(1)
\]
and, taking the limit as $\epsilon\to 0$,
\[
|\Omega| \leq \varphi(y_0) |\partial\Omega|.
\]
This inequality, together with (\ref{f:h1b}), ensures that the hypotheses of Theorem \ref{t:geom2} are satisfied. Hence $\Omega$ must be a ball.  \qed



\end{document}